\documentclass{elsarticle}
\usepackage{amsfonts}
\usepackage{amsmath}
\usepackage[mathscr]{eucal}
\usepackage{amssymb}
\usepackage{latexsym}
\usepackage{amsthm}
\usepackage{amscd}
\usepackage{epsf}
\usepackage{graphicx}
\usepackage{makeidx}
\makeindex
\newtheorem{theorem}{Theorem}[section]
\newtheorem{proposition}[theorem]{Proposition}
\newtheorem{corollary}[theorem]{Corollary}
\newtheorem{lemma}[theorem]{Lemma}

\newtheorem{definition}[theorem]{Definition}

\newtheorem{remark}[theorem]{Remark}

\newcommand{\bsa}{\boldsymbol{a}}

\newcommand{\bsk}{\boldsymbol{k}}
\newcommand{\bsl}{\boldsymbol{l}}

\newcommand{\bst}{\boldsymbol{t}}
\newcommand{\bsu}{\boldsymbol{u}}

\newcommand{\bsx}{\boldsymbol{x}}

\newcommand{\bsalpha}{\boldsymbol{\alpha}}

\newcommand{\Acal}{\mathcal{A}}

\newcommand{\dbsx}{\, d\bsx}

\newcommand{\nat}{\mathbb{N}}
\newcommand{\natzero}{\mathbb{N}_0}

\newcommand{\midmid}{\;\middle|\;}

\def\bN{\mathbb{N}}
\def\bZ{{\mathbb Z}}

\def\bR{{\mathbb R}}
\def\bC{{\mathbb C}}

\def\F2{{\mathbb F}_2}

\newcommand{\vol}{\mathrm{vol}}

\newcommand{\tr}{\mathrm{tr}}
\newcommand{\wal}{\mathrm{wal}}

\newcommand{\Zb}{\mathbb{Z}_b}

\newcommand{\logep}{\log {\varepsilon^{-1}}}
\newcommand{\logeppow}[1]{(\log {\varepsilon^{-1})^{#1}}}

\newcommand{\Dwt}[2][\bsa]{\mu(#1; #2)}
\newcommand{\mDwt}[2][\bsa]{\tilde{\mu}(#1; #2)}

\newcommand{\worsterr}{e^{\mathrm{wor}}}
\newcommand{\Ninfo}{n(\varepsilon, s)}
\newcommand{\NinfoS}[1][\Hspace]{n(\varepsilon, s, \Hspace)}
\newcommand{\minerr}[1][n]{e(#1,s)}
\newcommand{\minerrS}[1][\Hspace]{e(n,s,#1)}
\newcommand{\minerrNS}[1][\Wspace]{e(N,s,#1)}

\newcommand{\Sspace}[1][\bsu]{\mathcal{F}_{s, #1}}
\newcommand{\Snorm}[2][\bsu]{\| #2 \|_{\Sspace[#1]}}
\newcommand{\Wspace}[1][\bsa]{\mathcal{W}_{s, #1, b}}

\newcommand{\Wnorm}[2][\bsa]{\| #2 \|_{\Wspace[#1]}}
\newcommand{\mWspace}[1][\bsa]{\widetilde{\mathcal{W}}_{s, #1, b}}

\newcommand{\mWnorm}[2][\bsa]{\| #2 \|_{\mWspace[#1]}}
\newcommand{\Hspace}{\mathcal{H}}
\newcommand{\Hnorm}[1]{\|#1\|_{\Hspace}}
\newcommand{\mvol}[2][\bsa]{\mathrm{vol}_{s,#1}(#2)}
\newcommand{\NegSet}[1]{\mathcal{N}_{#1}}
\newcommand{\SumSet}{N_s} 
\newcommand{\SumSettrac}{N}
\newcommand{\Cconv}{C_s} %
\newcommand{\Cconvhelp}{C'_s}
\newcommand{\Ctracvol}{c_1} %
\newcommand{\Ctracbd}{c_2}
\newcommand{\Ctrachelp}{c_3}
\newcommand{\mb}{m_b}
\newcommand{\Mb}{M_b}
\newcommand{\CWalcoeff}{D_b} %

\begin{document}
\begin{frontmatter}
\title{Super-polynomial convergence and tractability of
multivariate integration for infinitely times differentiable functions}
\author{Kosuke Suzuki \fnref{fn1}}
\ead{kosuke.suzuki1@unsw.edu.au}
\fntext[fn1]{Present address: School of Mathematics and Statistics, The University of New South Wales, Sydney, NSW 2052, Australia}
\address{Graduate School of Mathematical Sciences, The University of Tokyo, 3-8-1 Komaba, Meguro-ku, Tokyo 153-8914 Japan}
\begin{abstract}
We investigate multivariate integration for a space of infinitely times differentiable functions
$
\Sspace := \{f \in C^\infty [0,1]^s \mid \Snorm{f} < \infty \},
$
where
$\Snorm{f} := \sup_{\bsalpha = (\alpha_1, \dots, \alpha_s) \in \natzero^s}
\|f^{(\bsalpha)}\|_{L^1}/\prod_{j=1}^s u_j^{\alpha_j}$,
$f^{(\bsalpha)} := \frac{\partial^{|\bsalpha|}}{\partial x_1^{\alpha_1} \cdots \partial x_s^{\alpha_s}}f$
and
$\bsu = \{u_j\}_{j \geq 1}$ is a sequence of positive decreasing weights.
Let $\minerr$ be the minimal worst-case error of all algorithms
that use $n$ function values in the $s$-variate case.
We prove that for any $\bsu$ and $s$ considered
$\minerr \leq C(s) \exp(-c(s)(\log{n})^2)$ holds for all $n$,
where $C(s)$ and $c(s)$ are constants which may depend on $s$.
Further we show that 
if the weights $\bsu$ decay sufficiently fast
then there exist some $1 < p < 2$ and absolute constants $C$ and $c$ such that
$\minerr \leq C \exp(-c(\log{n})^p)$ holds for all $s$ and $n$.
These bounds are attained by quasi-Monte Carlo integration using digital nets.
These convergence and tractability results come from those for the Walsh space
into which $\Sspace$ is embedded.
\end{abstract}

\end{frontmatter}

\section{Introduction}
In this paper we approximate the integral on an $s$-dimensional unit cube
\[
\int_{[0,1)^s} f(\bsx) \dbsx
\]
by a quasi-Monte Carlo (QMC) algorithm which uses $n$ function values of the form
\[
A_{n,s}(f) := \sum_{i=1}^n \frac{1}{n} f(\bst_i) \qquad \text{for $\bst_i \in [0,1)^s$}.
\]
One classical issue is the optimal rate of convergence with respect to $n$.
Another important issue is the dependence on the number of variables $s$,
since $s$ can be hundreds or more in computational applications.
The latter issue is related to the notion of tractability
if we require no exponential dependence on $s$.


A large number of studies have been devoted to numerical integration on the unit cube
for various function spaces.
One typical case is that functions are only finitely many times differentiable,
e.g., functions with bounded variation,
periodic functions in the Korobov space
and non-periodic functions in the Sobolev space,
see \cite{Niederreiter1992rng, Sloan1994lmm, Novak2010tmp, Dick2010dna} and the references therein.
For these cases, it is known that the rate of convergence is $O(n^{-\alpha})$
for some $\alpha > 0$ and thus we have polynomial convergence.
Another interesting case is when the functions are smooth, i.e., infinitely times differentiable.
Dick \cite{Dick2006tsm} gave reproducing kernel Hilbert spaces based on Taylor series
for which higher order QMC rules achieve a convergence of $O(n^{-\alpha})$ with $\alpha > 0$ arbitrarily large.
The spaces were later generalized in \cite{Zwicknagl2013iaa}.
Further results were proved in \cite{Dick2011eca, Kritzer2014mii}, 
where it is shown that exponential convergence holds
for the Korobov space of periodic functions
whose Fourier coefficients decay exponentially fast.
Exponential convergence means that the integration error converges as $O(q^{n^p})$
for some $q \in (0,1)$, $p>0$.
Note that exponential convergence was also shown for Hermite spaces on $\bR^s$ with exponentially fast decaying Hermite coefficients \cite{Irrgeher2015iih}.

In this paper we focus on a weighted normed space
of non-periodic smooth functions
\begin{equation}\label{eq:Sspace}
\Sspace := \left\{f \in C^\infty [0,1]^s \midmid \Snorm{f} := \sup_{\bsalpha = (\alpha_1, \dots, \alpha_s) \in \natzero^s} \frac{\|f^{(\bsalpha)}\|_{L^1}}{\prod_{j=1}^s u_j^{\alpha_j}} < \infty \right\}
\end{equation}
with a sequence of positive weights $\bsu = \{u_j\}_{j \geq 1}$,
where 
$f^{(\bsalpha)}
:= \frac{\partial^{|\bsalpha|}}{\partial x_1^{\alpha_1} \cdots \partial x_s^{\alpha_s}}f$.
It is easy to check that all functions in $\Sspace$ are analytic from Taylor's theorem.
For $s=1$, it is known
that the integration error using Gaussian quadrature with $n$ points converges factorially,
see, for instance, \cite[(6.53)]{Brass2011qt}.
Our interest is the multivariate QMC rules using so-called digital nets.
This is motivated by the results by Yoshiki \cite{Y}
and is closely related to the notion of
Walsh figure of merit (WAFOM)
\cite{Matsumoto2014acf, SuzukiMW, Y} first introduced by Matsumoto, Saito and Matoba.
WAFOM is a criterion for numerical integration
using digital nets and is computable in a reasonable time.
Hence we can search for good digital nets with respect to WAFOM by computer,
see \cite{Matsumoto2014acf, HOX, Harase2015qmc} for numerical experiments.
We observe that generalized WAFOM works well for the space $\Sspace$,
see Remark~\ref{rem:WAFOM}. 

The first purpose of this paper is to show that
the integration error of QMC rules using digital nets for $\Sspace$ can achieve super-polynomial convergence as $O(\exp({-c(\log{n})^2})) \asymp O(n^{-c\log{n}})$
for all $s$ and $\bsu$ considered.
Here the hidden constant and $c$ may depend on $s$.
We remark that this convergence behavior was first observed in \cite{Matsumoto2013eoh}
as the decay of the lowest-WAFOM value
and that the combination of \cite{Matsumoto2013eoh} and \cite{Y}
implies the convergence result for $\Sspace[(1/2)_{j \geq 1}]$.

We also consider tractability for $\Sspace$.
Let us briefly recall the notion of tractability
(see \cite{Novak2008tmp,Novak2010tmp, Novak2012tmp} for more information).
Let $\Ninfo$ be the information complexity, i.e.,
the minimal number $n$ of function values which are required in order to approximate
the $s$-variate integration within $\varepsilon$.
An integration problem is said to be
tractable if $\Ninfo$ does not grow exponentially in $\varepsilon$ nor $s$.
In particular, two notions of tractability have been mainly considered:
polynomial tractability, i.e., $\Ninfo \leq C \varepsilon^{-\tau_1} s^{\tau_2}$,
and strong polynomial tractability, i.e., $\Ninfo \leq C \varepsilon^{-\tau_1}$
for $\tau_1, \tau_2 \geq 0$.
A common way to obtain tractability is to consider weighted function spaces
as introduced by Sloan and Wo{\'z}niakowski \cite{Sloan1998waq}.
Weighted spaces here mean that the dependence on the successive variables can be
moderated by weights.
Our weights $\bsu$ play the same role.
For tractability results for spaces of smooth functions,
see also \cite{Hinrichs2014cdn}.

The second purpose of this paper is
to give a sufficient condition to achieve super-polynomial convergence with strong tractability.
We show that if the weights $\bsu$ decay sufficiently fast
then the integration error of QMC rules using digital nets for $\Sspace$ can 
achieve dimension-independent super-polynomial convergence as $O(\exp({-c(\log{n})^p}))$,
where the hidden constant and $c$ are independent of $s$ and $n$,
and $1 < p  <2$ is determined from the decay of  $\bsu$.
This implies $\Ninfo \leq C\exp(c\logeppow{1/p})$ for some $C, c \geq 0$.

These convergence and tractability results are also shown for the so-called Walsh space $\Wspace$
into which $\Sspace$ is embedded, which implies those for $\Sspace$.
For $\Wspace$, 
we show that the rate of convergence is of order $\exp(\Theta(-(\log n)^2))$
for all $\bsa$ considered
and the strong tractability result is equivalent to a typical decay of $\bsa$.

The rest of the paper is organized as follows.
In Section~\ref{sec:Walsh}, we give the necessary background
including Walsh functions, the Dick weight,
definitions of our function spaces and embeddings among them.
In Section~\ref{sec:integration}, we give precise definitions of the notions of
convergence and tractability used in this paper.
In Section~\ref{sec:results}, we present Theorems~\ref{thm:main}
and \ref{thm:main-Sspace},
which are the summary of all results in this paper.
Necessary and sufficient conditions for convergence and tractability
are given in Sections~\ref{sec:lower bounds} and \ref{sec:upper bounds},
respectively.


\section{Preliminaries}\label{sec:Walsh}
Throughout this paper, we shall use the following notation.
Let $\nat$ be the set of positive integers and $\natzero := \nat \cup \{0\}$.
For a positive integer $b\ge 2$, let $\bZ_b$ be a cyclic group with $b$ elements,
which we identify with the set $\{0,1,\dots,b-1\}$,
equipped with addition modulo $b$.
The operators $\oplus$ and $\ominus$ denote the digitwise addition and subtraction modulo $b$, respectively.
That is, for $k, k' \in\natzero$ whose $b$-adic expansions are
$k=\sum_{i=1}^{\infty}\kappa_i b^{i-1}$ and $k'=\sum_{i=1}^{\infty}\kappa'_i b^{i-1}$ with
$\kappa_i,\kappa'_i\in \Zb$ for all $i$, $\oplus$ and $\ominus$ are defined as
  \begin{align*}
    k\oplus k' = \sum_{i=1}^{\infty}\eta_i b^{i-1}\ \text{and}\ k\ominus k' = \sum_{i=1}^{\infty}\eta'_i b^{i-1},
  \end{align*}
where $\eta_i=\kappa_i+\kappa'_i \pmod b$
and $\eta'_i=\kappa_i-\kappa'_i \pmod b$, respectively.
In case of vectors in $\natzero^s$, the operators $\oplus$ and $\ominus$ are applied componentwise.

\subsection{Walsh functions}
In this subsection, we introduce Walsh functions and Walsh coefficients,
which are widely used in analyzing the integration error,
see \cite[Appendix~A]{Dick2010dna} for general information.
We first give the definition of Walsh functions for the one-dimensional case
and then generalize it to the higher-dimensional case.
\begin{definition}
Let $b \geq 2$ be a positive integer and let $\omega_b=\exp(2\pi \sqrt{-1}/b)$.
We denote the $b$-adic expansion of $k \in \natzero$ by
$k = \kappa_1 + \kappa_2 b + \dots  +\kappa_{i} b^{i-1}$
with $\kappa_1, \dots, \kappa_{i} \in \bZ_b$.
Then the {\it $k$-th $b$-adic Walsh function}
${}_b\wal_k \colon [0,1) \to \{1, \omega_b, \dots, \omega_b^{b-1}\}$ is defined as
  \begin{align*}
    {}_b\wal_k(x) := \omega_b^{\kappa_1 \xi_1 + \dots + \kappa_i \xi_i} ,
  \end{align*}
for $x\in [0,1)$ whose $b$-adic expansion is given by
$x=\xi_1b^{-1} + \xi_2b^{-2} + \cdots$,
which is unique in the sense that infinitely many of the $\xi_i$ are different from $b-1$.
\end{definition}

\begin{definition}
Let $b, s \in \bN$ with $b \geq 2$.
Let $\bsx=(x_1, \dots, x_s)\in [0,1)^s$ and $\bsk=(k_1, \dots, k_s)\in \natzero^s$.
The $\bsk$-th $b$-adic Walsh function
${}_b\wal_{\bsk} \colon [0,1)^s \to \{1,\omega_b, \dots, \omega_b^{b-1}\}$
is defined as
  \begin{align*}
    {}_b\wal_{\bsk}(\bsx) := \prod_{j=1}^s {}_b\wal_{k_j}(x_j) .
  \end{align*}
\end{definition}
Since we shall always use Walsh functions in a fixed base $b$,
we omit the subscript and simply write $\wal_k$ or $\wal_{\bsk}$ in this paper.
Some important properties of Walsh functions, used in this paper, are described below,
see \cite[Appendix~A.2]{Dick2010dna} for the proof.

\begin{proposition}\label{prop:walsh}
The following holds true:
\begin{enumerate}
\item For all $\bsk \in \natzero^s$, we have
  \begin{align*}
    \int_0^1 \wal_{\bsk}(\bsx) \dbsx = \begin{cases}
     1 & \text{if $\bsk=0$,}  \\
     0 & \text{otherwise.}
    \end{cases}
  \end{align*}
\item For all $\bsk, \bsl \in \natzero^s$, we have
  \begin{align*}
    \int_{[0,1)^s} \wal_{\bsk}(\bsx)\overline{\wal_{\bsl}(\bsx)} \dbsx
= \begin{cases}
     1 & \text{if $\bsk=\bsl$,}  \\
     0 & \text{otherwise.}
    \end{cases}
  \end{align*}
\item For all $\bsk, \bsk' \in \natzero^s$ and $\bsx \in [0,1)^s$, we have
\begin{align*}
\wal_{\bsk \oplus \bsk'}(\bsx) = \wal_{\bsk}(\bsx) \wal_{\bsk'}(\bsx), \\
\wal_{\bsk \ominus \bsk'}(\bsx) = \wal_{\bsk}(\bsx) \overline{\wal_{\bsk'}(\bsx)}.
\end{align*}
\item The system $\{\wal_{\bsk} \mid \bsk\in \natzero^s\}$ is a complete orthonormal system in $L^2[0,1)^s$ for any positive integer $s$.
\end{enumerate}
\end{proposition}

We define the Walsh coefficients as follows.
\begin{definition}
Let $\bsk \in \natzero^s$ and $f \colon [0,1)^s \to \bC$.
The {\it $\bsk$-th Walsh coefficient} of $f$ is defined as
  \begin{align*}
     \widehat{f}(\bsk) := \int_{[0,1)^s} f(\bsx)\overline{\wal_{\bsk}(\bsx)} \dbsx .
  \end{align*}
\end{definition}
\noindent
The Walsh series of the function $f$ is given by
  \begin{align*}
     f(\bsx) \sim \sum_{\bsk\in \natzero^s}\widehat{f}(\bsk)\wal_{\bsk}(\bsx)
  \end{align*}
for any $f\in L^2[0,1)^s$.
We note that all functions considered in this paper are equal to their Walsh series,
see also the next subsection.


\subsection{Function spaces and embeddings}
In this subsection, we introduce the function spaces $\Sspace$, $\Wspace$ and $\mWspace$
considered in this paper and give embeddings from $\Sspace$ to $\Wspace$.

The space of smooth functions $\Sspace$ is defined as in \eqref{eq:Sspace}.
Throughout the paper, we always assume that
\begin{equation}\label{eq:weight-Sspace-decrease}
u_1 \geq u_2 \geq \cdots > 0.
\end{equation}

It is shown in \cite{Y} for $b=2$ and \cite{SY} for the general case
that Walsh coefficients of functions in $\Sspace$
decay sufficiently fast, as given in Theorem~\ref{thm:Sspace-Wcoeff} below.
To state the theorem,
we define the Hamming weight $v(k)$,
the generalized Dick weight $\Dwt{\bsk}$ and the modified Dick weight $\mDwt{\bsk}$
for $k \in \natzero$, $\bsa \in \bR^s$ and $\bsk \in \natzero^s$.
Note that the Dick weight is originally defined for $\bsa = 0$ in \cite{Matsumoto2014acf}.
\begin{definition}
Let $a \in \bR$ and $k \in \bN$ with $b$-adic expansion
$k = \sum_{i=1}^\infty \kappa_{i} b^{i-1}$ with $\kappa_{i} \in \Zb$.
Define the function $h$ as $h(\kappa) = 0$ for $\kappa =0$ and $h(\kappa) = 1$ for $\kappa \neq 0$.
The {\it Hamming weight} $v(k)$ is defined as the number of nonzero digits for the $b$-adic expansion of $k$, i.e.,
\[
v(k) := \sum_{i=1}^{\infty} h(\kappa_{i}).
\]

We define the {\it generalized Dick weight} $\Dwt[a]{k}$ and
the {\it modified Dick weight} $\mDwt[a]{k}$ for the 1-dimensional case as
\begin{align*}
\Dwt[a]{k}
:= \sum_{i=1}^{\infty} (i+a)h(\kappa_{i})
\qquad \text{and} \qquad
\mDwt[a]{k}
:= \sum_{i=1}^{\infty} \max(i+a, 1) h(\kappa_{i}).
\end{align*}
For the $s$-dimensional case,
let $\bsa = (a_1, \dots, a_s) \in \bR^s$ and $\bsk = (k_1, \dots, k_s) \in \natzero^s$
and we define $\Dwt{\bsk}$ and $\mDwt{\bsk}$ as
\begin{align*}
\Dwt[\bsa]{\bsk}
:= \sum_{j=1}^{s} \Dwt[a_j]{k_j}
\qquad \text{and} \qquad
\mDwt[\bsa]{\bsk}
:= \sum_{j=1}^s \mDwt[a_j]{k_j}.
\end{align*}
\end{definition}
We can now give the decay of Walsh coefficients
which appears in \cite[Corollary~3.10]{SY}.
\begin{theorem}\label{thm:Sspace-Wcoeff}
Put $\mb:= 2\sin(\pi/b)$ and $\Mb := 2\sin(\lfloor b/2 \rfloor \pi/b)$.
Assume $f \in \Sspace$. Then it follows that
\[
|\widehat{f}(\bsk)|
\leq \Snorm{f} b^{-\Dwt[0]{\bsk}} \prod_{j=1}^s (\mb^{-1} u_j)^{v(k_j)} \CWalcoeff^{\min(1, v(k_j))},
\]
where 
$\CWalcoeff = 2$ for $b = 2$ and
$\CWalcoeff = \Mb + b\mb/(b - \Mb)$ otherwise.
\end{theorem}

This decay motivates us to define Walsh spaces $\Wspace$ and $\mWspace$
of Walsh series whose Walsh coefficients are controlled
by the generalized (resp.\ modified) Dick weight.
Let $\bsa = (a_j)_{j \geq 1}$ be a sequence of real-valued weights.
Throughout the paper, we assume
\begin{equation}\label{eq:weight-increasing}
a_1 \leq a_2 \leq a_3 \leq \cdots,
\end{equation}
which corresponds to \eqref{eq:weight-Sspace-decrease}.
We first define $\Wspace$ as
\[
\Wspace
:= \left\{f\colon [0,1)^s \to \bR \midmid f(\bsx) = \sum_{\bsk \in \natzero^s} \widehat{f}(\bsk) \wal_{\bsk}(\bsx) \,\text{ and }\, \Wnorm{f} < \infty \right\}
\]
equipped with the norm
\[
\Wnorm{f} := \sup_{\bsk \in \natzero^s} |\widehat{f}(\bsk) b^{\Dwt{\bsk}}|
\]
and $\mWspace$ as
\[
\mWspace
:= \left\{f \in \Wspace \midmid \mWnorm{f} := \sup_{\bsk \in \natzero^s} |\widehat{f}(\bsk) b^{\mDwt{\bsk}}| < \infty \right\}.
\]
\noindent
Note that all Walsh series in $\Wspace$ and $\mWspace$ converge.
Indeed, for all $X \in (-1, 1)$ and a positive integer $l$, we have
\begin{align*}
\sum_{\substack{\bsk \in \natzero^s \\ k_j < b^l \, \forall j}} X^{\Dwt{\bsk}}
&= \sum_{\substack{\bsk \in \natzero^s \\ k_j < b^l \, \forall j}} \prod_{j=1}^s\prod_{i=1}^l X^{(i+a_j)h(\kappa_{j,i})}\\
&= \prod_{j=1}^s\prod_{i=1}^l \sum_{\kappa_{j,i}=0}^{b-1} X^{(i+a_j)h(\kappa_{j,i})}\\
&= \prod_{j=1}^s \prod_{i=1}^l (1+(b-1)X^{i+a_j}),
\end{align*}
where we denote the $b$-adic expansion of $k_j$ by
$k_j = \sum_{i=1}^l \kappa_{j,i}b^{i-1}$ with $\kappa_{j,i} \in \bZ_b$
in the first equality,
and the right-most product converges for $l \to \infty$ if $|X| < 1$.
This is also true for the modified Dick weight
with $\Dwt{\bsk}$ and $i+a_j$
replaced by $\mDwt{\bsk}$ and $\max(i+a_j, 1)$.
Hence we have
\begin{align}\label{eq:power-series}
\sum_{\bsk \in \natzero^s} X^{\Dwt{\bsk}}
&= \prod_{j=1}^s \prod_{i=1}^\infty (1+(b-1)X^{i+a_j})& &\text{for all} \quad |X| < 1,\\
\label{eq:power-series-modified}
\sum_{\bsk \in \natzero^s} X^{\mDwt{\bsk}}
&= \prod_{j=1}^s \prod_{i=1}^\infty (1+(b-1)X^{\max(i+a_j, 1)})
& &\text{for all} \quad |X| < 1.
\end{align}
Thus all functions in $\Wspace$ and $\mWspace$ converge.

We now give embeddings from $\Sspace$ to $\Wspace$.
From Theorem~\ref{thm:Sspace-Wcoeff} we have
\begin{align*}
|\widehat{f}(\bsk)|
&\leq \Snorm{f} \prod_{j=1}^s \CWalcoeff^{\min(1, v(k_j))} b^{-\Dwt[-\log_b(\mb^{-1} u_j)]{\bsk}} \\
&\leq \Snorm{f} \prod_{j=1}^s b^{-\Dwt[-\log_b(\CWalcoeff \mb^{-1} u_j)]{\bsk}}
\end{align*}
for $f \in \Sspace$. Thus we obtain continuous embeddings
\begin{align}
\Sspace
&\subset \Wspace[\bsu']
& &\text{with}
& \Wnorm[\bsu']{f}
&\leq \Snorm{f} \label{eq:S-W-embed}
,\\
\Sspace
&\subset \Wspace[\bsu'']
& &\text{with}
& \Wnorm[\bsu'']{f}
&\leq \CWalcoeff^s \Snorm{f},
\end{align}
where
$\bsu' = (-\log_b(\CWalcoeff \mb^{-1} u_j))_{j \geq 1}$
and 
$\bsu'' = (-\log_b(\mb^{-1} u_j))_{j \geq 1}$.
Note that all functions in $\Sspace$ are equal to their Walsh expansions,
see \cite[Section~3.3]{Dick2008wsc} or \cite[Theorem~A.20]{Dick2010dna}.
Embedding \eqref{eq:S-W-embed} implies that
good algorithms for $\Wspace[\bsu']$ are also good for $\Sspace$.
Thus we mainly consider $\Wspace$ in the following sections.

The Walsh space $\mWspace$ is considered instead of $\Wspace$
in Section~\ref{sec:upper bounds},
since the modified Dick weight does not take negative values
and thus is easier to treat.
Actually, $\Wspace$ and $\mWspace$ are 
norm equivalent.
Indeed, we have
\begin{align*}
\Dwt{\bsk}
\leq \mDwt{\bsk}
\leq \Dwt{\bsk}
+ \sum_{j=1}^s \sum_{i \in \NegSet{j}} (1 - (i + a_j))
\end{align*}
for all $\bsk \in \natzero^s$,
where $\NegSet{j}$ is defined as $\NegSet{j} := \{i \in \nat \mid i + a_j \leq 1 \}$.
Thus
\begin{equation}\label{eq:W-mW-embed}
\Wnorm{f} \leq \mWnorm{f} \leq b^{\sum_{j=1}^s \sum_{i \in \NegSet{j}} (1 - (i + a_j))}\Wnorm{f},
\end{equation}
where the empty sum equals 0, which implies the norm-equivalence of $\Wspace$ and $\mWspace$.
This means that
we can consider $\mWspace$ instead of $\Wspace$
for convergence results.
Furthermore, in Section~\ref{sec-upper-trac}, where we consider tractability results,
we shall assume a condition on the weights which implies
that the constant factor in \eqref{eq:W-mW-embed} is bounded independently of $s$.


\section{Integration}\label{sec:integration}
Let $\Hspace = \Sspace$, $\Wspace$ or $\mWspace$.
We consider multivariate integration
\[
I(f) = \int_{[0,1)^s} f(\bsx) \dbsx \qquad \text{for all} \quad f \in \Hspace.
\]
We approximate $I(f)$ by algorithms
\[
A_{n,s}(f) = \phi_{n,s}(f(\bst_1), \dots, f(\bst_n)), 
\]
where $\phi_{n,s}\colon \bR^n \to \bR$ is an arbitrary function and $\bst_i \in [0,1)^s$ for $i = 1, \dots, n$.
The {\it worst-case error} of the algorithm $A_{n,s}$ in the space $\Hspace$ is defined by
\[
\worsterr(A_{n,s}, \Hspace)
= \sup_{\substack{f \in \Hspace \\ \Hnorm{f} \leq 1}} |I(f) - A_{n,s}(f)|.
\]
\noindent
Let $\minerrS$ be the {\it $n$-th minimal worst-case error},
\[
\minerr
= \minerrS
= \inf_{A_{n,s}} \worsterr(A_{n,s}, \Hspace),
\]
where the infimum is extended over all algorithms using $n$ function values.
For $n=0$, we approximate $I(f)$ by a real number.
Since $\Hspace$ is symmetric, i.e., $f \in \Hspace$ implies $-f \in \Hspace$,
the zero algorithm is the best for $n=0$, and thus we have $e(0,s,\Hspace)=1$.
Hence the integration problem is well normalized for all $s$.

For $\varepsilon \in (0,1)$, we define the {\it information complexity} of integration
\[
\Ninfo = \NinfoS = \min \{n \in \nat \mid \minerrS \leq \varepsilon\}
\]
as the minimal number of function values needed to obtain an $\varepsilon$-approximation.

We are interested in the convergence of the minimal worst-case error of the form
\begin{equation}\label{eq:acc-conv}
\minerr \leq C(s) e^{-c(s) (\log{n})^{p}} \qquad \text{for all $s, n \in \nat$},
\end{equation}
where $C(s)$ and $c(s)$ are positive real numbers which may depend on $s$
and where $p>1$. The condition $p>1$ implies that this convergence is super-polynomial.
We note that if \eqref{eq:acc-conv} holds,
then for all $s\in \nat$ and $\varepsilon \in (0,1)$ we have
\begin{equation}\label{eq:acc-to-info}
\Ninfo \leq
\left\lceil \exp\left(\left(\frac{\log{C(s)} + \logep}{c(s)}\right)_+^{1/p}\right) \right\rceil,
\end{equation}
where $(X)_+ := \max(X,0)$ for $X \in \bR$.
Furthermore, 
\eqref{eq:acc-to-info} for $\varepsilon = C(s) e^{-c(s) (\log{n})^{p}}$ implies \eqref{eq:acc-conv}.
This means that \eqref{eq:acc-conv} is equivalent to \eqref{eq:acc-to-info},
which shows that
asymptotically $\Ninfo$ increases with order $\exp (O(\logeppow{1/p}))$
with respect to $\varepsilon$.
However, how does $\Ninfo$ depend 
on $s$?
This, of course, depends on $C(s)$ and $c(s)$ and
is the subject of tractability.
Tractability means that we control the behavior of $C(s)$ and $c(s)$
and rule out the cases for which $\Ninfo$ depends exponentially on $s$.
In this paper we consider two convergence behaviors as
\begin{equation}\label{eq:ACPT-conv}
\minerr \leq C \exp(As)  e^{-c (\log{n})^{p}} \qquad \text{for all $s, n\in \nat$},
\end{equation}
and
\begin{equation}\label{eq:ACST-conv}
\minerr \leq C e^{-c (\log{n})^{p}} \qquad \text{for all $s, n\in \nat$},
\end{equation}
for some $p > 1$ and $A, C, c>0$.
Applying \eqref{eq:acc-to-info} to \eqref{eq:ACPT-conv}
and using the inequality
$(X+Y)^{1/p} \leq X^{1/p} + Y^{1/p}$
for $X,Y \geq 0$,
\eqref{eq:ACPT-conv} implies that 
\begin{equation}\label{eq:AC-PT}
\Ninfo \leq C' \exp(c'\logeppow{1/p}) \exp(A's^{1/p})
\qquad \text{for all} \quad s \in \nat, \, \varepsilon \in (0,1)
\end{equation} 
for some $A', C', c' \geq 0$.
Conversely, \eqref{eq:AC-PT} implies \eqref{eq:ACPT-conv} for some $A, C, c \geq 0$,
which follows from the inequality $2^{-1+1/p}(X^{1/p} + Y^{1/p}) \leq (X+Y)^{1/p}$ for $X,Y \geq 0$.
Similarly \eqref{eq:ACST-conv} is equivalent to the fact that 
\begin{equation}\label{eq:AC-ST}
\Ninfo \leq C' \exp(c'\logeppow{1/p})
\qquad \text{for all} \quad s \in \nat, \, \varepsilon \in (0,1)
\end{equation} 
for some $C', c' \geq 0$,
which can be regarded as super-polynomial convergence with strong tractability.
Although \eqref{eq:AC-PT} shows that the information complexity may depend super-polynomially on $s$,
we will show that \eqref{eq:ACPT-conv} is equivalent to \eqref{eq:ACST-conv} for the Walsh space.



\section{Main results}\label{sec:results}
In this section, we present the main results of this paper.
The following theorems give the super-polynomial convergence and tractability results
for $\Wspace$ and $\Sspace$.
\begin{theorem}\label{thm:main}
Consider integration defined over the Walsh space $\Wspace$
with a weight sequence $\bsa$ satisfying \eqref{eq:weight-increasing}.
Then we have the following.
\begin{enumerate}
\item \label{item:main-UAC}
For fixed $\bsa$, there exist positive constants $c_{i,s}$ $(i=1,2,3,4)$
which may depend on $s$ and $\bsa$ such that it holds that 
\begin{align*}
\exp \left(-\frac{(\log{n})^2}{2s \log b} - c_{1,s}\log{n} - c_{2,s} \right)
&\leq
\minerrS[\Wspace]\\
&\leq c_{3,s} \exp(-c_{4,s} (\log{n})^2)
\end{align*}
for all $s$ and $n$.
In particular $\minerrS[\Wspace]$ is of order $\exp(\Theta(-(\log n)^2))$.

\item\label{item:main-AC-PT}
For any $\bsa$ considered, there do not exist $A, C, c \geq 0$ such that
\eqref{eq:ACPT-conv} holds for $p=2$.

\item \label{item:main-equiv}
Let $1<p<2$.
Then the following are equivalent.
\begin{enumerate}
\item The sequence $\bsa$ satisfies $\liminf_{j \to \infty} a_j/j^{(p-1)/(2-p)} > 0$, \label{item:main3-1}
\item There exist constants $A,C,c \geq 0$ such that for all $s, n\in \nat$ we have \label{item:main3-2}
\[
\minerr \leq C \exp(As)  e^{-c (\log{n})^{p}}.
\]
\item There exist constants $C,c \geq 0$ such that for all $s, n\in \nat$ we have \label{item:main3-3}
\[
\minerr \leq C e^{-c (\log{n})^{p}}.
\]
\end{enumerate}
\end{enumerate}
\end{theorem}

\begin{theorem}\label{thm:main-Sspace}
Consider integration defined over $\Sspace$
with a weight sequence $\bsu$ satisfying \eqref{eq:weight-Sspace-decrease}.
Then we have the following.
\begin{enumerate}
\item There exist positive constants $c_{5,s}$ and $c_{6,s}$ depending on $s$ and $\bsu$ such that \label{item:maincor-UAC}
\[
\minerrS[\Sspace] \leq c_{5,s}\exp(-c_{6,s} (\log{n})^2) \qquad \text{for all $s, n\in \nat$}.
\]
\item Let  $1 < p < 2$ be a real number. 
If the weight sequence $\bsu$ satisfies
$\liminf_{j \to \infty} \log(u_j^{-1})/j^{(p-1)/(2-p)} > 0$,
then there exist constants $C,c \geq 0$ such that \label{item:maincor-ACST}
\[
\minerr \leq C e^{-c (\log{n})^{p}} \qquad \text{for all $s, n\in \nat$}.
\]
\end{enumerate}
\end{theorem}

These theorems follow from Theorems \ref{thm:UAC-nes}--\ref{thm:ACPT-nes2}
and Corollaries~\ref{cor:conv-result} and \ref{cor:trac-result}.


\section{Lower bounds}\label{sec:lower bounds}
We prove the following lower bound on $\minerrS[\Wspace]$
similarly to \cite[Theorem~1]{Dick2011eca},
which treats the Korobov space.
\begin{lemma}
Let $\Acal$ be a finite subset of $\natzero^s$.
Then for all $n < |\Acal|$ we have
\[
\minerrS[\Wspace] \geq \left(\max_{\bsk, \bsk^* \in \Acal} b^{\Dwt{\bsk \ominus \bsk^*}}\right)^{-1}.
\]
\end{lemma}

\begin{proof}
Take an arbitrary algorithm $A_{n,s}(f) = \phi_{n,s}(f(\bst_1), \dots, f(\bst_n))$.
Define $g_1(\bsx) = \sum_{\bsk \in \Acal} c_{\bsk} \wal_{\bsk}(\bsx)$
for $c_{\bsk} \in \bC$
such that $g_1(\bst_i) = 0$ for all $i=1,2, \dots, n$.
Since we have $n$ homogeneous linear equations and $|\Acal| > n$
unknowns $c_{\bsk}$, there exists a nonzero vector of such $c_{\bsk}$'s,
and we can normalize the $c_{\bsk}$'s by assuming that
\[
\max_{\bsk \in \Acal} |c_{\bsk}| = c_{\bsk^*} = 1
\quad \text{for some} \quad \bsk^*\in \Acal.
\]
Define the function
\[
g_2(\bsx) := Cg_1(\bsx) \overline{\wal_{\bsk^*}(\bsx)}
= C\sum_{\bsk \in \Acal} c_{\bsk} \wal_{\bsk \ominus \bsk^*}(\bsx),
\]
where $C$ is defined as $C := (\max_{\bsk, \bsk^* \in \Acal} b^{\Dwt{\bsk \ominus \bsk^*}})^{-1}$.
Then we have
\begin{align*}
\Wnorm{g_2}
&= C \max_{\bsk \in \Acal} |c_{\bsk} b^{\Dwt{\bsk \ominus \bsk^*}}| \\
&\leq C \max_{\bsk \in \Acal} b^{\Dwt{\bsk \ominus \bsk^*}}
\leq C \max_{\bsk, \bsk^* \in \Acal} b^{\Dwt{\bsk \ominus \bsk^*}}
=1,
\end{align*}
where $\Wnorm{\cdot}$ is naturally extended to complex-valued Walsh series.

We now define a real-valued function
$
f(\bsx) := (g_2(\bsx) + \overline{g_2}(\bsx))/2,
$
where $\overline{g_2}(\bsx) := \overline{g_2(\bsx)}$.
Note that $\Wnorm{\overline{g_2}} = \Wnorm{g_2}$
since $\Dwt{\bsk \ominus \bsk^*} = \Dwt{\bsk^* \ominus \bsk}$
for all $\bsk$.
The norm of $f$ is bounded by 
\[
\Wnorm{f} \leq (\Wnorm{g_2} + \Wnorm{\overline{g_2}})/2 = \Wnorm{g_2} \leq 1.
\]
We note that $A_{n,s}(f) = \phi_{n,s}(0, \dots, 0)$ and $I(f) = Cc_{\bsk^*} = C$.
Hence,
\begin{align*}
\minerrS[\Wspace] \geq |I(f) - A_{n,s}(f)| \geq|C - \phi_{n,s}(0, \dots, 0)|.
\end{align*}
Further we consider the function $-f$. We have $\Wnorm{-f} \leq 1$ and $A_{n,s}(-f) = \phi_{n,s}(0, \dots, 0)$.
Hence,
\begin{align*}
\minerrS[\Wspace] \geq |I(-f) - A_{n,s}(-f)| \geq|C + \phi_{n,s}(0, \dots, 0)|.
\end{align*}
Combining these two inequalities, we have 
\begin{align*}
\minerrS[\Wspace] \geq \max(|C - \phi_{n,s}(0, \dots, 0)|, |C + \phi_{n,s}(0, \dots, 0)|) \geq C.
\end{align*}
Since this holds for arbitrary algorithm $A_{n,s}$, we conclude that $\minerr \geq C$,
as claimed.
\end{proof}

For a non-negative integer $d$, we now define
\[
\Acal_{s, d}
=\{\bsk \in \natzero^s \mid k_j < b^{d} \quad \text{for all} \quad j = 1,2,\dots, s \}.
\]
The cardinality of the set $|\Acal_{s, d}|$ is clearly $b^{sd}$.
If $a_j \geq 0$ holds for all $j$, then
\begin{align*}
\left(\max_{\bsk, \bsk^* \in \Acal_{s, d}} b^{\Dwt{\bsk \ominus \bsk^*}}\right)^{-1}
&= \left(\max_{\bsk \in \Acal_{s, d}} b^{\Dwt{\bsk}}\right)^{-1}\\
&= b^{-\sum_{j=1}^s \sum_{i=1}^{d} (i+a_j)}
= b^{-\sum_{j=1}^s (d(d + 1)/2 + a_j d)},
\end{align*}
where we use $\bsk \ominus \bsk^* \in \Acal_{s, d}$
for all $\bsk, \bsk^* \in \Acal_{s, d}$ for the first equality.
This implies the following corollary.

\begin{corollary}\label{cor:lower-bound}
Let $d \in \nat$ and
assume $a_j \geq 0$ for all $j$.
Then we have
\[
\minerrS[\Wspace]
\geq b^{-\sum_{j=1}^s (d^2/2 + (a_j + 1/2)d)} \qquad \text{for all $n < b^{sd}$}.
\]
\end{corollary}

We prove a lower bound for the worst-case error and a necessary condition to achieve \eqref{eq:ACPT-conv} for $\Wspace$
in the following three theorems with the following notation.
For $x \in \bR$, we define $\exp_b(x) := b^x$ 
and $\max(\bsa, x) := (\max(a_j, x))_{j=1}^s$.

\begin{theorem}\label{thm:UAC-nes}
Let $\bsa = (a_j)_{j=1}^s \in \bR^s$ and $\bsa' = (a'_j)_{j=1}^s := \max(\bsa,0)$.
Then we have
\[
\minerrS[\Wspace]
\geq \exp_b \left(-\frac{(\log{n})^2}{2s (\log b)^2} - \left(\frac{3s}{2}+\sum_{j=1}^s a'_j\right)\frac{\log{n}}{s\log{b}} - \left(s+\sum_{j=1}^s a'_j \right) \right).
\]
\end{theorem}

\begin{proof}
Let $n \in \nat$.
Put $d = \lfloor \log{n}/(s \log{b})\rfloor + 1$,
so that $b^{s(d-1)} \leq n < b^{sd}$.
It follows from Corollary~\ref{cor:lower-bound} and the embedding $\Wspace[\bsa'] \subset \Wspace$ that
\begin{align*}
\minerrS[\Wspace]
&\geq \minerrS[{\Wspace[\bsa']}]
\geq b^{-\sum_{j=1}^s (d^2/2 + (a'_j + 1/2)d)}\\
&\geq \exp_b\left(-{\frac{s}{2}\left(\frac{\log{n}}{s\log{b}}+1\right)^2 - \left(\frac{\log{n}}{s\log{b}}+1\right)\sum_{j=1}^s (a'_j + 1/2)}\right)\\
&\geq \exp_b \left(-\frac{(\log{n})^2}{2s (\log b)^2} - \left(\frac{3s}{2}+\sum_{j=1}^s a'_j\right)\frac{\log{n}}{s\log{b}} -s -\sum_{j=1}^s a'_j  \right),
\end{align*}
which proves the result.
\end{proof}

\begin{theorem}\label{thm:ACPT-nes1}
For any $\bsa$ considered, there do not exist $A, C, c \geq 0$ such that
\eqref{eq:ACPT-conv} holds for $p=2$.
\end{theorem}

\begin{proof}
We will argue by contradiction.
Let $\bsa' = \max(\bsa, 0)$.
Suppose that \eqref{eq:ACPT-conv} holds for some $A, C, c \geq 0$.
Then this assumption  and Theorem~\ref{thm:UAC-nes} imply that
\begin{align*}
\exp_b \left(-\frac{(\log{n})^2}{2s (\log b)^2} - \left(\frac{3s}{2}+\sum_{j=1}^s a'_j\right)\frac{\log{n}}{s\log{b}} - \left(s+\sum_{j=1}^s a'_j \right) \right)\\
\leq \minerrS[{\Wspace[\bsa']}]
\leq \minerrS[\Wspace]
\leq C \exp(As)  e^{-c (\log{n})^{2}}
\end{align*}
holds for all $s$ and $n$.
Taking the limit as $n$ goes to infinity, we have
$1/(2s \log b) \geq c$ for all $s$.
This is a contradiction.
\end{proof}

\begin{theorem}\label{thm:ACPT-nes2}
Consider integration defined over $\Wspace$
under \eqref{eq:weight-increasing}.
Let $1<p<2$ and assume that
\eqref{eq:ACPT-conv} holds for this $p$ and some $A, C, c \geq 0$.
Put $r := (p-1)/(2-p)$.
Then we have
\[
\liminf_{j \to \infty} \frac{a_j}{j^r} > 0.
\]
\end{theorem}

\begin{proof}
Let $\bsa' = \max(\bsa, 1)$ and $N = N(s) := b^{s \lfloor a'_s \rfloor} - 1$.
Corollary~\ref{cor:lower-bound} implies
\begin{align*}
\minerrNS[{\Wspace[\bsa']}]
\geq \exp_b\left(-\sum_{j=1}^s\left(\frac{\lfloor a'_s \rfloor^2}{2} + \left(a'_j + \frac{1}{2}\right)\lfloor a'_s \rfloor \right)\right)
\geq \exp_b(-3s\lfloor a'_s \rfloor^2),
\end{align*}
where we use $a'_j + 1/2 \leq a'_s + 1/2 \leq 5\lfloor a'_s \rfloor/2$
in the second inequality.
Combining this with the assumption that \eqref{eq:ACPT-conv} holds for $\Wspace$, 
for all $s$ we have
\[
\exp_b(-3s\lfloor a'_s \rfloor^2)
\leq C \exp(As)  e^{-c (\log{N(s)})^{p}}.
\]
By taking the logarithm and using $b^{(s-1)\lfloor a'_s \rfloor} \leq N(s)$ and $C \leq \exp(Cs)$, we have 
\[
-3s (\log{b}) \lfloor a'_s \rfloor^2
\leq (C + A)s  - c ((s-1)\lfloor a'_s \rfloor\log{b})^{p},
\]
and thus 
\[
\frac{-(C+A)}{s^{p-1}\lfloor a'_s \rfloor^p}  + \frac{(s-1)^p}{s^p}c(\log{b})^{p}
\leq 3 \log{b} \left(\frac{\lfloor a'_s \rfloor}{s^r}\right)^{2-p}. 
\]
holds for all $s$.
Taking the limit inferior as $s$ goes to infinity,
we have
\[
\liminf_{j \to \infty} \frac{a'_j}{j^r} \geq \left(\frac{c(\log{b})^{p-1}}{3}\right)^{1/(2-p)} > 0,
\]
which implies the desired result.
\end{proof}


\section{Upper bounds}\label{sec:upper bounds}
In this section, motivated by \cite{Matsumoto2013eoh} and its generalization \cite{SuzukiMW},
we prove the existence of good QMC algorithms
which achieve super-polynomial convergence and tractability
in Sections~\ref{sec-upper-conv} and \ref{sec-upper-trac},
respectively.
Such QMC algorithms are given by digital nets.
Digital nets are point sets which have the structure
of a $\Zb$-module introduced by Niederreiter, see for instance \cite{Niederreiter1992rng}.
We introduce the notion of digital nets in the following subsection.

\subsection{Digital nets}
For a positive integer $m$ and a non-negative integer $k$ with its $b$-adic expansion
$k = \sum_{i=1}^\infty \kappa_{i}b^{i-1}$,
we define the $m$-digit truncated vector $\tr_m(k) \in \Zb^m$
as $\tr_m(k) = (\kappa_1, \kappa_2, \dots, \kappa_m)^{\top}$.

\begin{definition}
Let $G_1, \dots, G_s\in \bZ_b^{l\times d}$ be $l \times d$ matrices
over $\bZ_b$ with $d \leq l$.
Let $0\le k<b^d$. 
For $1\leq j \leq s$ and $1 \leq i \leq l$, define $y_{i,k,j} \in \Zb$ as
\[
(y_{1,k,j},\dots,y_{l,k,j})^{\top} = G_j\tr_d(k),
\]
where the matrix vector multiplication is over $\Zb$.
Then we define
  \begin{align*}
    x_{k,j}=\frac{y_{1,k,j}}{b}+\frac{y_{2,k,j}}{b^2}+\dots+\frac{y_{l,k,j}}{b^l} \in [0,1)
  \end{align*}
for $1\le j\le s$.
In this way we obtain the $k$-th point $\bsx_k=(x_{k,1}, \dots, x_{k,s})$.
We define $P=P(G_1, \dots, G_s) := \{\bsx_0,\dots,\bsx_{b^d-1}\}$
($P$ is considered as a multiset) and call it a {\it digital net over $\Zb$ with precision $l$}, or simply a {\it digital net}.
\end{definition}

The {\em dual net} of a digital net, which is defined as follows,
plays an important role in the subsequent analysis.
\begin{definition}\label{def:dual_net}
Let $P=P(G_1, \dots, G_s)$
be a digital net over $\Zb$ with precision $l$.
The dual net of $P$,
denoted by $P^\perp = P^\perp(G_1, \dots, G_s)$, is defined as
\begin{align*}
P^\perp :=
\{\bsk = (k_1, \dots, k_s) \in \natzero^s \mid
G_1^{\top} \tr_l(k_1) + \cdots + G_s^{\top} \tr_l(k_s) = 0\}.
\end{align*}
\end{definition}

The next lemma, which is a slight generalization of \cite[Lemma~4.75]{Dick2010dna} to our context, connects a digital net with Walsh functions.
\begin{lemma}\label{lem:dual_Walsh}
Let $P$ be a digital net over $\Zb$ and $P^\perp$ its dual net. Then we have
  \begin{align*}
     |P|^{-1}\sum_{\bsx \in P}\wal_{\bsk}(\bsx) = 
\begin{cases}
     1 & \text{if} \ \bsk \in P^\perp,  \\
     0 & \text{otherwise} . 
    \end{cases} 
  \end{align*}
\end{lemma}

From now on, we consider integration defined over $\mWspace$.
We use QMC algorithms over digital nets.
That is, for a digital net $P$, we use
$P(f) := |P|^{-1} \sum_{\bsx \in P} f(\bsx)$,
where we identify the digital net $P$ and the QMC algorithm on $P$.
The (signed) integration error of $f \in \mWspace$ by $P$ is calculated as
\begin{align*}
|P|^{-1} \sum_{\bsx \in P} f(\bsx) - I(f)
&= |P|^{-1} \sum_{\bsx \in P} \sum_{\bsk \in \natzero^s}\widehat{f}(\bsk) \wal_{\bsk}(\bsx)  - I(f)\\
&= \sum_{\bsk \in \natzero^s}\widehat{f}(\bsk) |P|^{-1}\sum_{\bsx \in P}\wal_{\bsk}(\bsx)  - I(f)\\
&= \sum_{\bsk \in P^\perp}\widehat{f}(\bsk) - \widehat{f}(0)\\
&= \sum_{\bsk \in P^\perp \backslash\{0\}}\widehat{f}(\bsk).
\end{align*}
Hence we have
\begin{equation}\label{eq:err-of-dig-net}
\left||P|^{-1} \sum_{\bsx \in P} f(\bsx) - I(f)\right|
\leq \sum_{\bsk \in P^\perp \backslash\{0\}}|\widehat{f}(\bsk)|
\leq \mWnorm{f} \sum_{\bsk \in P^\perp \backslash\{0\}} b^{-\mDwt{\bsk}}.
\end{equation}

\begin{remark}\label{rem:WAFOM}
WAFOM, a criterion for digital nets, is defined as a truncated version of the sum
on the rightmost side of \eqref{eq:err-of-dig-net} for $\bsa = 0$
in \cite{Matsumoto2014acf,SuzukiMW}
and for $\bsa = 1$ in \cite{Y, Harase2015qmc}.
Note that $\mDwt{\bsk} = \Dwt{\bsk}$ in these cases.
Thus the sum (and the sum with $\mDwt{\bsk}$ replaced by $\Dwt{\bsk}$)
can be regarded as an untruncated version of WAFOM
generalized by weights $\bsa$,
which has not been considered as far as the author knows. 
We can say that $\Sspace$, $\Wspace$ and $\mWspace$ are function spaces
for which WAFOM is a suitable quality criterion.
\end{remark}

We now define the minimal weight of $P^\perp$ by
\begin{equation*}
\delta_{P^\perp} := \inf_{\bsk \in P^\perp \backslash \{0\}}{\mDwt{\bsk}}.
\end{equation*}
Then the rightmost side of \eqref{eq:err-of-dig-net}
is bounded by
$
\mWnorm{f} \sum_{\bsk} b^{-\mDwt{\bsk}},
$
where the sum is extended over 
all $\bsk \in \natzero^s$ with $\mDwt{\bsk} \geq \delta_{P^\perp}$.
This argument implies the following lemma.
\begin{lemma}\label{lem:werr-of-nets}
Let $P$ be a digital net. Then we have
\begin{equation}\label{eq:minDwt-bounds-err}
\worsterr(P, \mWspace)
\leq \sum_{\substack{\bsk \in \natzero^s \\ \mDwt{\bsk} \geq \delta_{P^\perp}}} b^{-\mDwt{\bsk}}.
\end{equation}
\end{lemma}
The right-hand side of \eqref{eq:minDwt-bounds-err}
will be evaluated in the following sections.

We now prove a lemma which gives the existence of digital nets
whose minimal weight is large,
which generalizes \cite[Proposition~2]{Matsumoto2013eoh} and \cite[Proposition~4]{SuzukiMW},
First we define
\[
\mvol{M}
:= |\{\bsk \in \natzero^s \mid \mDwt{\bsk} \leq M\}|.
\]

\begin{lemma}\label{lem:exist-min-wt}
Let $M$ be a real number and
${\rho_b}$ be the smallest prime factor of $b$.
Let $d, l$ be positive integers with $l \geq M - a_1 - 1$.
If $\mvol{M} \leq {\rho_b}^d$ holds,
then there exists a digital net $P$ over $\Zb$ with precision $l$ 
satisfying $|P| = b^d$  and $\delta_{P^\perp} \geq M$.
\end{lemma}

\begin{proof}
Let $G_1, \dots G_s \in \Zb^{l \times d}$ be matrices.
Recall
\[
\bsk \in P^\perp(G_1, \dots, G_s)
\iff G_1^{\top} \tr_l(k_1) + \cdots + G_s^{\top} \tr_l(k_s) = 0.
\]
Thus, for a given $0 \neq \bsk \in \bN_0^s$ with $k_j < b^l$ for all $j$, we have
\[
|\{(G_1, \dots, G_s) \in (\Zb^{l \times d})^s
\mid \bsk \in P^\perp(G_1, \dots, G_s) \}|
\leq b^{sdl}/{\rho_b}^d,
\]
where ${\rho_b}$ is the smallest prime factor of $b$.
Hence we have
\begin{align*}
&|\{(G_1, \dots, G_s) \in (\Zb^{l \times d})^s \mid
\min_{\substack{\bsk \in P^\perp(G_1, \dots, G_s)\backslash\{0\} \\ k_j<b^l \, \forall j}} \mDwt{\bsk} > M\}|\\
&= b^{sdl} - |\{(G_1, \dots, G_s) \in (\Zb^{l \times d})^s \mid
\min_{\substack{\bsk \in P^\perp(G_1, \dots, G_s)\backslash\{0\} \\ k_j<b^l \, \forall j}} \mDwt{\bsk} \leq M\}|\\
&\geq b^{sdl} - \sum_{\substack{\bsk \neq 0, \,\, k_j<b^l \, \forall j \\ \mDwt{\bsk} \leq M}}
|\{(G_1, \dots, G_s) \in (\Zb^{l \times d})^s \mid
\bsk \in P^\perp(G_1, \dots, G_s) \}|\\
&> b^{sdl} - \mvol{M} b^{sdl}/{\rho_b}^d.
\end{align*}
Thus, if  $\mvol{M} \leq {\rho_b}^d$ holds,
there exists $(G_1, \dots, G_s) \in (\Zb^{l \times d})^s$ with
\begin{equation}\label{eq:exist-min-wt-step1}
\inf_{\substack{\bsk \in P^\perp(G_1, \dots, G_s)\backslash\{0\} \\ k_j < b^l \, \forall j}}  \mDwt{\bsk} \geq M.
\end{equation}
Furthermore, from the assumption $l \geq M - a_1 - 1$ we have
\begin{equation}\label{eq:exist-min-wt-step2}
\min\{\mDwt{\bsk} \mid \bsk \in \natzero^s, k_j \geq b^l \,\, \exists j\}
\geq \max(1, a_1 + l + 1)
\geq M.
\end{equation}
Combining \eqref{eq:exist-min-wt-step1} and \eqref{eq:exist-min-wt-step2},
we obtain the result.
\end{proof}

\subsection{Super-polynomial convergence results}\label{sec-upper-conv}
In this subsection, we prove super-polynomial convergence
for $\Sspace$, $\Wspace$ and $\mWspace$.
Taking \eqref{eq:S-W-embed} and \eqref{eq:W-mW-embed} into account, we have only to consider $\mWspace$.

First we prove a bound on $\mvol{M}$ along \cite[Exercise~3(b), p.332]{Matouvsek1998itd}
and its modifications \cite{Matsumoto2013eoh, SuzukiMW},
which treat the case of $\bsa = 0$.
Since $\mvol{M} \leq 1$ holds if $M < 1$, we assume that $M \geq 1$.
We have
\begin{align*}
\mvol{M}
= \sum_{\substack{\bsk \in \natzero^s \\ \mDwt{\bsk}\leq M}} 1
\leq \sum_{\substack{\bsk \in \natzero^s \\ \mDwt{\bsk}\leq M}} X^{\mDwt{\bsk} - M}
\leq \sum_{\bsk \in \natzero^s} X^{\mDwt{\bsk} - M}
\end{align*}
for all $X \in (0,1)$,
and the right-most expression is equal to
$\prod_{j=1}^s \prod_{i=1}^\infty (1+(b-1)X^{\max(i+a_j, 1)})/X^M$
from \eqref{eq:power-series}.
By taking the logarithm on both sides
and using the well-known inequality $\log(1+X) \leq X$,
for all $X \in (0,1)$ we have
\begin{align}\label{eq:vol-firststep}
\log \mvol{M}
&\leq \sum_{j=1}^s \sum_{i=1}^\infty (b-1)X^{\max(i+a_j, 1)} + M \log{X^{-1}}.
\end{align}

We proceed to bound $\sum_{i=1}^\infty X^{\max(i+a_j, 1)}$.
If $a_j \geq 0$, it is equal to $X^{a_j+1}/(1-X)$.
Otherwise, we have
\begin{align*}
\sum_{i=1}^\infty X^{\max(i+a_j, 1)}
&= \sum_{i \colon i + a_j \leq 1} X^1 + \sum_{i \colon i+a_j > 1} X^{i+a_j}\\
&\leq \sum_{i \colon i + a_j \leq 1} 1 + \sum_{i'=1}^\infty X^{i'}\\
&= n_j + X/(1-X),
\end{align*}
where $n_j := |\NegSet{j}| = |\{i \in \nat \mid i + a_j \leq 1\}|$.
Thus, in both cases, we obtain
\[
\sum_{i=1}^\infty X^{\max(i+a_j, 1)}
\leq n_j + \frac{X}{1-X} \min(X^{a_j}, 1).
\]
Applying this inequality to \eqref{eq:vol-firststep}, we have 
\begin{align}
\log \mvol{M}
&\leq (b-1)\sum_{j=1}^s \left(n_j + \frac{X}{1-X} \min(X^{a_j}, 1)\right) + M \log{X^{-1}} \notag \\
&\leq (b-1)\sum_{j=1}^s \left(n_j + (\log{X^{-1}})^{-1} \min(X^{a_j}, 1)\right) + M \log{X^{-1}} \label{eq:vol-secondstep}.
\end{align}

Putting $X = 1/\exp(\sqrt{(b-1)s/M})$
and using $\min(X^{a_j}, 1) \leq 1$, we obtain
\begin{align*}
\log \mvol{M}
\leq \SumSet + 2\sqrt{(b-1)sM},
\end{align*}
where we define $\SumSet := (b-1)\sum_{j=1}^s n_j$.
We have thus proved the following.

\begin{lemma}\label{lem:vol-conv}
For all $M \geq 0$ we have
\[
\mvol{M}
\leq \exp \left(\SumSet + 2\sqrt{(b-1)sM}\right).
\]
\end{lemma}

We note that Lemma~\ref{lem:vol-conv} and the fact that $\mvol{M} \leq 1$ if $M < 1$
implies
\begin{equation} \label{eq:vol-conv2}
\mvol{M}
\leq \exp \left(\left(\SumSet + 2\sqrt{(b-1)s}\right)\sqrt{M}\right).
\end{equation}

Now we give a bound on
the right-hand side of \eqref{eq:minDwt-bounds-err}.
From Lemma~\ref{lem:vol-conv} we have
\begin{align}
\sum_{\substack{\bsk \in \natzero^s \\ \mDwt{\bsk} \geq M}} b^{-\mDwt{\bsk}}
&\leq \sum_{i=0}^\infty \sum_{\substack{\bsk \in \natzero^s \\ M+i \leq \mDwt{\bsk} < M+i+1}} b^{-(M+i)} \notag \\
&\leq \sum_{i=0}^\infty \vol_{s, \bsa}(M+i+1) b^{-(M+i)} \notag \\
&\leq \sum_{i=0}^\infty \exp\left(\SumSet + 2\sqrt{(b-1)s(M+i+1)}\right) b^{-(M+i)} \label{eq:conv-bd}
\end{align}
for all $M \geq 0$.
We can easily check
$
\sqrt{x}
\leq x/(2\sqrt{B}) + \sqrt{B}/2
$
for all $x, B \geq 0$.
Applying this inequality with $x = M + i + 1$,
the right-hand side of \eqref{eq:conv-bd} is bounded by
\begin{align*}
& \sum_{i=0}^\infty \exp\left(\SumSet + \sqrt{(b-1)s/B}(M+i+1) + \sqrt{(b-1)sB}\right) b^{-(M+i)}\\
&= b\exp\left(\SumSet + \sqrt{(b-1)sB} \right) \sum_{i=0}^\infty \exp\left(\left(\sqrt{(b-1)s/B} - \log{b}\right)(M+i+1) \right).
\end{align*}
Taking $B$ as 
$
\sqrt{(b-1)s/B} = (\log{b})/2,
$
we obtain a bound on the right-hand side of \eqref{eq:minDwt-bounds-err} by
\begin{align*}
\sum_{\substack{\bsk \in \natzero^s \\ \mDwt{\bsk} \geq M}} b^{-\mDwt{\bsk}}
\leq \Cconv \exp(-(\log{b})M/2),
\end{align*}
where the positive constant $\Cconv$ is defined by
\begin{align*}
\Cconv
&= \exp\left(\SumSet + (\log{b})/2 + 2(b-1)s/\log{b}\right)
(1-\exp(-(\log{b})/2))^{-1}.
\end{align*}
Hence Lemma~\ref{lem:werr-of-nets} implies the following lemma.
\begin{lemma}\label{thm:minwt-to-WAFOM-conv}
Let $P$ be a digital net.
Then we have
\[
\worsterr(P, \mWspace)
\leq \Cconv \exp(-\delta_{P^\perp}(\log{b})/2).
\]
\end{lemma}

Put
$\Cconvhelp := \SumSet + 2\sqrt{(b-1)s}$.
From \eqref{eq:vol-conv2}, the condition of Lemma~\ref{lem:exist-min-wt}
is satisfied if 
$\exp(\Cconvhelp \sqrt{M}) \leq \rho_b^d$,
which is equivalent to
$M \leq (d\log{\rho_b}/\Cconvhelp)^2$.
Therefore the following bound on the worst-case error follows
from Lemmas \ref{lem:exist-min-wt} and \ref{lem:vol-conv}.

\begin{theorem}\label{thm:existence-goodnet-conv}
Let $d$ be a positive integer.
Then there exists a digital net $P$ over $\Zb$ with precision $l$
with $|P| = b^d$ and $l \geq (\log{\rho_b}/\Cconvhelp)^2d^2-1-a_1$ such that
\begin{align}\label{eq:conv-theorem}
\worsterr(P, \mWspace)
\leq \Cconv \exp\left(-C''_s d^2\right),
\end{align}
where $C''_s := (\log{\rho_b})^2(\log{b})/(2{\Cconvhelp}^2)$.
\end{theorem}

In particular, $e(b^d, s, \mWspace)$ is bounded by the right-hand side of \eqref{eq:conv-theorem}.
Therefore we have $\minerrS[\mWspace] \leq \Cconv \exp\left(-C''_s \lfloor\log_b{n}\rfloor^2 \right)$.
Thus embeddings \eqref{eq:S-W-embed} and \eqref{eq:W-mW-embed}
imply the following convergence result.
\begin{corollary}\label{cor:conv-result}
There exist constants $C_{i,s}$ and $C'_{i,s}$ ($i=1,2,3$) which depend on $s$ and the weights $\bsu$ or $\bsa$
such that for all $n$ we have
\begin{align*}
\minerrS[\mWspace] \leq C_{1,s} \exp\left(-C'_{1,s} (\log{n})^2 \right),\\
\minerrS[\Wspace] \leq C_{2,s} \exp\left(-C'_{2,s} (\log{n})^2 \right),\\
\minerrS[\Sspace] \leq C_{3,s} \exp\left(-C'_{3,s} (\log{n})^2 \right).
\end{align*}
\end{corollary}


\subsection{Tractability results}\label{sec-upper-trac}
We have proved super-polynomial convergence for the function spaces, 
but this convergence depends heavily on $s$.
In this subsection, we prove a tractability result
under the assumption 
of the sufficient condition from Theorem~\ref{thm:main}.
That is, let $r > 0$ and 
assume that
the sequence $\bsa$ satisfies $\liminf_{j \to \infty} a_j/j^r > 0$.
This implies that there exist a positive real number $a$ and a non-negative integer $A$
such that
\begin{equation}\label{eq:weight-assumption}
a_j \geq aj^r \qquad \text{for all} \quad j > A.
\end{equation}
Hence hereafter we assume \eqref{eq:weight-assumption}.
Under this assumption,
$\NegSet{j} = \{i \in \nat \mid i + a_j \leq 1\}$ is empty for sufficiently large $j$.
Hence the constant factor in \eqref{eq:W-mW-embed} is independent of $s$ and thus
the tractability result for $\Wspace$ follows from that for $\mWspace$.
We also note that $(b-1)\sum_{j=1}^\infty n_j$ is finite and we denote it by
$\SumSettrac$.
We also show Item~\ref{item:maincor-ACST} of Theorem~\ref{thm:main-Sspace}
using the embedding \eqref{eq:S-W-embed}.
The following arguments are parallel to those in Section~\ref{sec-upper-conv}.

First we prove a bound on $\mvol{M}$ under the assumption \eqref{eq:weight-assumption}.
We need the following lemma to bound $\sum_{j=1}^s X^{a_j}$. 
\begin{lemma}\label{lem:sum of power of X}
For all $0 < X < 1$, we have
\[
\sum_{j=1}^s X^{aj^r} \leq  r^{-1}\Gamma(1/r) (a\log{X^{-1}})^{-1/r},
\]
where $\Gamma(z) := \int_0^\infty t^{z-1}\exp(-t) \, dt$ is the Gamma function.
\end{lemma}

\begin{proof}
Since $X^{ax^r}$ is a monotonically decreasing function of $x$, we have
\[
\sum_{j=1}^s X^{aj^r} \leq \int_0^s X^{ax^r} \, dx \leq \int_0^\infty \exp(-ax^r\log{X^{-1}}) \, dx.
\]
Now we consider the substitution $ax^r \log{X^{-1}} = z$. Then we have
$dx = r^{-1} (a\log{X^{-1}})^{-1/r} z^{(1-r)/r} \, dz$,
and thus
\begin{align*}
\int_0^\infty \exp(-ax^r \log{X^{-1}}) \, dx
&= r^{-1}(a\log{X^{-1}})^{-1/r} \int_0^\infty z^{(1-r)/r}\exp(-z) \, dz\\
&= r^{-1}\Gamma(1/r) (a\log{X^{-1}})^{-1/r},
\end{align*}
which proves the lemma.
\end{proof}

Combining \eqref{eq:vol-secondstep} and Lemma~\ref{lem:sum of power of X},
for all $X \in (0,1)$ we have
\begin{align*}
\log \mvol{M}
&\leq (b-1)\left(\sum_{j=1}^{A}\left(n_j + \frac{1}{\log{X^{-1}}}\right)  + \sum_{j=A+1}^s \frac{X^{ar^j}}{\log{X^{-1}}} \right) + M \log{X^{-1}} \\
&\leq (b-1)\left(\frac{A}{\log{X^{-1}}} + \frac{r^{-1} \Gamma(1/r)a^{-1/r}}{(\log{X^{-1})^{1+1/r}}} \right) + \SumSettrac + M \log{X^{-1}}.
\end{align*}
Putting $X = 1/\exp(M^{-r/(2r+1)})$ and using $M \geq 1$, we obtain
\begin{align*}
\log \mvol{M}
&\leq \Ctracvol M^{(r+1)/(2r+1)},
\end{align*}
where
$\Ctracvol
= (b-1)(A + r^{-1} \Gamma(1/r)a^{-1/r}) + \SumSettrac + 1$.
We have thus proved the following lemma.
\begin{lemma}\label{lem:vol}
Assume \eqref{eq:weight-assumption}. Then for all $M \geq 0$ we have
\[
\mvol{M} \leq \exp\left(\Ctracvol M^{(r+1)/(2r+1)} \right).
\]
\end{lemma}
Note that the bound on $\mvol{M}$ from this lemma is weaker than
Lemma~\ref{lem:vol-conv}
with respect to $M$ but independent of $s$ instead.


In the following, we bound
the right-hand side of \eqref{eq:minDwt-bounds-err}
along the lines of Section~\ref{sec-upper-conv}.
For $M \geq 0$ we have
\begin{align}
\sum_{\substack{\bsk \in \natzero^s \\ \mDwt{\bsk} \geq M}} b^{-\mDwt{\bsk}}
\leq \sum_{i=0}^\infty \exp(\Ctracvol (M+i+1)^{(r+1)/(2r+1)}) b^{-(M+i)} \label{eq:trac-bd}.
\end{align}
We can easily check the inequality
\[
x^{(r+1)/(2r+1)}
\leq \frac{r+1}{2r+1}B^{-r} x + \frac{r}{2r+1}B^{r+1}
\qquad \text{for all $x, B \geq 0$}.
\]
Applying this inequality with $x = M + i + 1$, the right-hand side of \eqref{eq:trac-bd}
is bounded by
\begin{align*}
& \sum_{i=0}^\infty \exp \left(\Ctracvol\frac{r+1}{2r+1}B^{-r}(M+i+1) + \Ctracvol\frac{r}{2r+1}B^{r+1}\right) b^{-(M+i)}\\
&= b\exp \left(\Ctracvol \frac{r}{2r+1}B^{r+1}\right)
\sum_{i=0}^\infty \exp \left(\left(\Ctracvol\frac{r+1}{2r+1}B^{-r} - \log{b}\right)(M+i+1)\right).
\end{align*}
Now we choose $B$ such that 
\[
\Ctracvol\frac{r+1}{2r+1}B^{-r} = \frac{\log{b}}{2}.
\]
Thus we have a bound on the right-hand side of \eqref{eq:minDwt-bounds-err} as
\begin{align*}
\sum_{\substack{\bsk \in \natzero^s \\ \mDwt{\bsk} \geq M}} b^{-\mDwt{\bsk}}
\leq \Ctracbd \exp(-M(\log{b})/2),
\end{align*}
where the positive constant $\Ctracbd$ is defined as
\begin{align*}
\Ctracbd
=\exp \left(\frac{\log{b}}{2} + \frac{\Ctracvol r}{2r+1}\left(\frac{2\Ctracvol (r+1)}{(2r+1)\log{b}}\right)^{(r+1)/r}\right) \frac{1}{1 - \exp(-(\log{b})/2)}.
\end{align*}
Hence Lemma~\ref{lem:werr-of-nets} implies the following lemma.
\begin{lemma}\label{thm:minwt-to-WAFOM-trac}
Assume \eqref{eq:weight-assumption}.
If $P$ is a digital net, we have
\begin{equation*}
\worsterr(P, \mWspace)
\leq \Ctracbd \exp(-\delta_{P^\perp}(\log{b})/2).
\end{equation*}
\end{lemma}

Now we prove the existence of good digital nets.
By Lemma~\ref{lem:vol}, the condition of Lemma~\ref{lem:exist-min-wt} is satisfied if 
$\exp(\Ctracvol M^{(r+1)/(2r+1)}) \leq {\rho_b}^d$,
which is equivalent to $M \leq (d\log{\rho_b}/\Ctracvol)^{(2r+1)/(r+1)}$.
Therefore we have
the following bound on the worst-case error independent of $s$.
\begin{theorem}\label{thm:existence-goodnet-trac}
Let $d \in \bN$ and put
$\Ctrachelp = ((\log{\rho_b})/\Ctracvol)^{(2r+1)/(r+1)}$.
Assume \eqref{eq:weight-assumption}.
Then there exists a digital net $P$ over $\Zb$ with precision $l$
with $|P| = b^d$ and $l \geq \Ctrachelp d^{(2r+1)/(r+1)}-1-a_1$ such that
\begin{align}\label{eq:trac-theorem}
\worsterr(P, \mWspace)
\leq \Ctracbd \exp\left(-\frac{\Ctrachelp \log{b}}{2} d^{(2r+1)/(r+1)}\right).
\end{align}
\end{theorem}
In particular, $e(b^d, s)$ is bounded by the right-hand side of \eqref{eq:trac-theorem}.
Therefore we have
$\minerrS[\mWspace] \leq \Ctracbd \exp\left(-c_4 \lfloor\log_b{n}\rfloor^{(2r+1)/(r+1)} \right)$,
where $c_4 := \Ctrachelp (\log{b})/2$.
Thus embeddings \eqref{eq:S-W-embed} and \eqref{eq:W-mW-embed}
imply the following tractability result.
\begin{corollary}\label{cor:trac-result}
Assume \eqref{eq:weight-assumption}.
Then there exist constants $C_{i}$ and $C'_{i}$ ($i=1,2$) which are independent of $s$
such that for all $s$ and $n$ we have
\begin{align*}
\minerrS[\mWspace] \leq C_{1} \exp\left(-C'_{1} (\log{n})^{(2r+1)/(r+1)} \right),\\
\minerrS[\Wspace] \leq C_{2} \exp\left(-C'_{2} (\log{n})^{(2r+1)/(r+1)} \right).
\end{align*}
Assume that the weight sequence $\bsu$ satisfies $\liminf_{j \to \infty} \log(u_j^{-1})/j^{r} > 0$.
Then there exists constants $C_3, C'_3$ independent of $s$ such that for all $s$ and $n$ we have
\[
\minerrS[\Sspace] \leq C_{3} \exp\left(-C'_{3} (\log{n})^{(2r+1)/(r+1)} \right).
\]
\end{corollary}

\section*{Acknowledgments}
This work was supported by the Program for Leading Graduate Schools, MEXT, Japan.
The author is grateful to Josef Dick for many helpful discussions and comments and a suggestion to focus on $\Sspace$ rather than the Walsh space.
The author would also like to thank Takashi Goda and Takehito Yoshiki
for valuable discussions and comments.

\bibliography{suzuki.bib}

\end{document}